\documentclass{birkjour}

\usepackage{amsmath,amssymb,amsfonts,amsthm,amscd}
\usepackage{mathrsfs,latexsym,url,graphicx,setspace,color}
\usepackage{hyperref}

\newtheorem{thm}{Theorem}[section]

 \theoremstyle{definition}
 
 \theoremstyle{remark}

 \numberwithin{equation}{section}

\newcommand{\abs}[1]{\left\lvert#1\right\rvert}
\newcommand{\norm}[1]{\lvert\lvert#1\rvert\rvert}

\newcommand{\disk}{\mathbb{D}}

\newcommand{\SOm}{\mathbf{S}_{\Omega}}

\begin{document}

\title[Szeg\"o Projection on Bounded Pseudoconvex Domains]{Irregularity of the Szeg\"o Projection on Bounded Pseudoconvex Domains in $\mathbb{C}^2$}

\author[S. Munasinghe]{Samangi Munasinghe}
\address{Department of Mathematics, Western Kentucky University, Bowling Green, Kentucky, 42101}
\email{samangi.munasinghe@wku.edu}

\author[Y. E. Zeytuncu]{Yunus E. Zeytuncu}
\address{Department of Mathematics and Statistics, University of Michigan-Dearborn, Dearborn, MI 48128}
\email{zeytuncu@umich.edu}

\date{}

\begin{abstract}
We construct bounded pseudoconvex domains in $\mathbb{C}^2$ for which the Szeg\"o projection operators are unbounded on $L^p$ spaces of the boundary for all $p\not =2$.
\end{abstract}


\maketitle 

\section{Introduction}
Let $\Omega$ be a bounded domain with piecewise smooth boundary in $\mathbb{C}^n$. Let $d\sigma_E$ denote the Euclidean surface measure on the boundary $b\Omega$, the topological boundary of $\Omega$,  and let $L^p(b\Omega)$ denote the standard Lebesgue space with respect to $d\sigma_E$ on $b\Omega$. For $1\leq p <\infty$, the Hardy space $H^p(b\Omega)$ is the closure of the set of holomorphic functions that are continuous up to the boundary in $L^p(b\Omega)$. When $p = 2$, there exists the orthogonal projection operator, so called the Szeg\"o projection operator,

$$\mathbf{S}_{\Omega} : L^2(b\Omega, d\sigma_E) \to H^2(b\Omega)$$
where
$$\mathbf{S}_{\Omega}f(z_1,z_2)=\int_{b\Omega}\mathbb{S}_{\Omega}[(z_1,z_2),(t_1,t_2)]f(t_1,t_2)d\sigma_E(t_1,t_2).$$
The integral kernel, $\mathbb{S}_{\Omega} [(z_1, z_2), (t_1, t_2)]$, is called the Szeg\"o kernel for
$H^2(b\Omega)$. For details of the setup, we refer the reader to \cite{SteinBigBook} for domains with smooth enough boundary and to \cite{Knopf} for piecewise smooth boundary like the bidisc.

It is a natural question to investigate the behavior of the integral operator $\SOm$ on $L^p(b\Omega)$ for $p\not= 2$. It is known that under certain geometric conditions on the domain $\Omega$ the projection operator $\SOm$ is bounded on $L^p(b\Omega)$ for all $p \in (1,\infty)$, see for example \cite{PhongStein, McNealSzego, Peloso, Charp06, LanzaniStein13}. On the other hand, $\SOm$ does not have to be bounded on all $L^p(b\Omega)$ spaces. However, explicit examples of bounded pseudoconvex domains where $\mathbf{S}_{\Omega}$ is unbounded on $L^p(b\Omega)$ for all $p\not =2$ is a new phenomena.  
On the other hand, in the complex plane $\mathbb{C}$, there are domains with piecewise smooth boundary whose Szeg\"o projections are bounded for some (not for all $p\not= 2$) of $L^p(b\Omega)$ (cf. \cite{Bekolle86, LanzaniStein04}). In the Hardy space setting corresponding to the Shilov boundary, it has been known for a long time that the associated Szeg\"o projection is unbounded for all $p\not= 2$ on irreducible bounded symmetric domains whose rank is greater than 1 (cf. \cite{BekolleSzego}).

In this note, we present examples of domains for which the Szeg\"o projection operators are $L^p$-irregular. More specifically, we construct bounded pseudoconvex domains in $\mathbb{C}^2$ (can be easily generalized to higher dimensions too) for which the Szeg\"o projection operators are unbounded on $L^p$ spaces for all $p\not= 2$. The construction follows the ideas in \cite{Zeytuncu13} with modifications from Bergman kernels to Szeg\"o kernels. Two main ingredients are the Forelli-Rudin inflation principle for Szeg\"o kernels \cite{Ligocka89, Englis} and the $L^p$ regularity of weighted Bergman projections on planar domains.

We work with the Euclidean surface measure $d\sigma_E$ on $b\Omega$ (the topological boundary\footnote{Some authors work on the distinguished boundary of the domain, instead of the
topological boundary.}) and hence with the Szeg\"o projection operator associated to $d\sigma_E$ . It is possible to put different measures on $b\Omega$ (or on the distinguished boundary), similar to Fefferman measure (see \cite{Barrett06}) or Leray measure (see \cite{LanzaniStein12, LanzaniStein13}), and investigate the corresponding function theory. However, the domains constructed in this note are only piecewise smooth and contain weakly pseudoconvex points. Therefore, it is not clear how to define such measures that transform well under biholomorphic maps. We postpone the discussion of different measures and investigation of associated Hardy spaces and Szeg\"o projections to a future paper.

\section{Statement and Proof} 

Let $\mathbb{D}$ be the unit disc in $\mathbb{C}$ and
\begin{displaymath}
\phi(z) = \left( 1 - \abs{z}^2 \right)^A \exp\left( \frac{-B}{\left( 1 - \abs{z}^2 \right)^{\alpha}} \right)
\end{displaymath}
for some $A\geq 0$, $B > 0$, and $\alpha > 0$. By using $\phi$, we define the following complete Reinhardt domain in $\mathbb{C}^2$,
\begin{displaymath}
\Omega = \left\{ (z_1, z_2) : z_1 \in \disk, |z_2| < \phi(z_1) \right\} .
\end{displaymath}
Note that for our choices of $A,B$ and $\alpha$ the function $-\log\phi(z)$ is subharmonic on $\disk$ and consequently the corresponding domain $\Omega$ is pseudoconvex, see \cite[page 129]{VlaBook}.

Note also that a holomorphic function $f$ on $\Omega$ is in $H^p(b\Omega)$ if 
\begin{eqnarray} 
\underset{U \subset \subset \disk}{\sup} ~~ \underset{\epsilon < \underset{U}{\inf} \phi}{\sup} \underset{z_1 \in U, |z_2|= \phi(z_1) - \epsilon}{\int} \abs{f}^p d\sigma_E < \infty. 
\end{eqnarray}

As indicated in \cite{Ligocka89, Englis, Knopf}, such functions have non-tangential
boundary values almost everywhere on $b\Omega$ and
\begin{eqnarray}
\norm{f}^p_{H^p} = \int_{b\Omega} \abs{f}^p d\sigma_E
\end{eqnarray}
where $d\sigma_E$ is explicitly written down by parametrizing the three real dimensional boundary and computing the Jacobian determinant (see \cite{Ligocka89, Englis})

\begin{displaymath}
\int_{b\Omega} \Phi ~d\sigma_E = \int_{\disk} \int_0^{2\pi} \Phi\left(z_1, e^{i\theta}\phi(z_1)  \right) \phi(z_1)  \sqrt{1 + \abs{\nabla \phi(z_1)}^2~} ~d\theta ~ dA(z_1).
\end{displaymath}

The generalized Forelli-Rudin construction indicates the following representation  (see \cite{Ligocka89, Englis}),
\begin{eqnarray} \label{SzegoDecomposition}
\mathbb{S}_{\Omega} [(z_1 , z_2 ), (t_1 , t_2 )] = \sum_{j = 0}^{\infty} z_2^j B_j (z_1 , t_1 )\overline{t^j_2}
\end{eqnarray}
where each $B_j(z_1,t_1)$ is the weighted Bergman kernel of the weight Bergman space
$$A^2 \left(\mathbb{D}, c_j \phi^{2j + 1}  \sqrt{1 + \abs{\nabla \phi}^2~} \right)$$
 (and $\mathbf{B}_j$ labels the corresponding weighted Bergman projection) for some
constant $c_j$. For the rest of the note, we set
\begin{displaymath}
\mu_j :=c_j \phi^{2j+1} \sqrt{1+ \abs{\nabla \phi}^2}.
\end{displaymath}

\begin{thm}
Let $\Omega$ be as above. The Szeg\"o projection operator ${\bf{S}}_{\Omega}$, associated to the Euclidean surface measure, is bounded on $L^p(b\Omega)$ if and only if $p=2$.
\end{thm}

\begin{proof}
It is clear that the operator is bounded when $p = 2$. We prove that it is unbounded for all $p \neq 2$.
We use the representation \eqref{SzegoDecomposition} to show that if the weighted Bergman projection ${\bf{B}}_0$ is unbounded on $L^{p_0} (\disk, \mu_0)$ for some $p_0 > 2$, then ${\bf{S}}_{\Omega}$ is also unbounded on $L^{p_0} (b\Omega)$.

In particular, let $\displaystyle{f(z_1) \in L^{p_0} (\disk, \mu_0)}$ for some $p_0 >2$ and let $F(z_1,z_2)= f(z_1)$ on $\Omega$. Then,
\begin{align*}
\norm{F}&^{p_0}_{L^{p_0}(b\Omega, d\sigma)}= \int_{b\Omega} \abs{F(z_1, z_2)}^{p_0} ~ d\sigma_E\\
&= \int_{\disk} \int_0^{2\pi} \abs{F(z_1, e^{i\theta} \phi(z_1))}^{p_0} \phi(z_1) \sqrt{1 + \abs{\nabla \phi(z_1)}^2} ~d\theta ~dA(z_1)\\
&= \int_{\disk} \int_0^{2\pi} \abs{f(z_1)}^{p_0} \phi(z_1) \sqrt{1 + \abs{\nabla \phi(z_1)}^2} ~d\theta ~dA(z_1)\\
&= 2 \pi \int_{\disk} \abs{f(z_1)}^{p_0} \phi(z_1) \sqrt{1 + \abs{\nabla \phi(z_1)}^2} ~dA(z_1)\\
&= 2 \pi \norm{f}^{p_0}_{L^{p_0}(\disk, \mu_0)}
\end{align*}

On the other hand,
\begin{align*}
{\bf{S}}_{\Omega}F (z_1, z_2) &= \int_{b\Omega} F(t_1, t_2) \left( \sum_{j = 0}^{\infty} z^j_2 B_j(z_1, t_1) \overline{t^j_2} \right) ~ d\sigma_E(t_1, t_2)\\
&= \int_{\disk} \int_0^{2\pi} F(t_1, t_2)  \left( \sum_{j = 0}^{\infty} z^j_2 B_j(z_1, t_1) \overline{t^j_2} \right) ~d\theta  \mu_0dA(t_1)\\
&= 2\pi \int_{\disk} f(t_1) B_0(z_1, t_1) \mu_0 dA(t_1)\\
&= 2\pi {\bf{B}}_0 f(z_1).
\end{align*}
Therefore, it remains to show that the weighted Bergman projection operator ${\bf{B}}_0$ is unbounded on $L^{p_0} (\disk, \mu_0)$. For this purpose, we invoke the following theorem \cite[Theorem 1.2]{Zeytuncu13} (see also \cite{Dostanic04}) for the weighted Bergman projection ${\bf{B}}_0$.

\begin{thm}\label{DZ}
If $\lambda$ is a radial weight on $\disk$ which satisfies \footnote{Here, we abuse the notation and consider $\lambda$ as a function on $[0, 1]$ and by $\lambda(z)$ we mean $\lambda(|z|)$.}
\begin{enumerate}
\item $\lambda(r)$ is a smooth function on $[0, 1]$,
\item $\frac{d^n}{dr^n}\lambda(1) = 0$ for all $n \in \mathbb{N}$, 
\item for all $n \in \mathbb{N}$ there exists $a_n \in (0, 1)$ such that $(-1)^n\frac{d^n}{dr^n}\lambda(r)$ is non-negative on the interval $(a_n, 1)$.
\end{enumerate}
Then the weighted Bergman projection $\mathbf{B}_{\lambda}$ is bounded from $L^p(\disk, \lambda)$ to $L^p(\disk, \lambda)$ if and only if $p = 2$.
\end{thm}

Recall that
\begin{displaymath}
\mu_0 = c_0\phi \sqrt{1 + \abs{\nabla \phi}^2}.
\end{displaymath}
We have to check that
$\mu_0$ is a radially symmetric function and, as a function of
$r=|z|$, it is smooth on $[0,1]$.  Moreover, $\frac{d^n}{dr^n} \mu_0(1) = 0$ for all $n \in \mathbb{N}$ 
and $(-1)^n \frac{d^n}{dr^n} \mu_0(r)$ is non-negative on $(a_n,1)$ for some $0 < a_n < 1$ for all $n \in \mathbb{N}$. 

We go over this verification for the particular case $A=0, B=1$ and $\alpha=1$; i.e. we take $\phi(z)=\exp\left(\frac{-1}{1-|z|^2}\right)$. Then a quick calculation gives (we drop $c_0$),

\begin{align*}
\mu_0(z)=\exp\left(\frac{-1}{1-|z|^2}\right)\sqrt{1+\exp\left(\frac{-2}{1-|z|^2}\right)\frac{2|z|^2}{(1-|z|^2)^4}}.
\end{align*}

It is clear that $\mu_0$ is indeed a radial function and, as a function of $r=|z|,$ it is smooth on $[0, 1].$ Furthermore we have 
$$\mu_0 (r) = \nu \left(\frac {1}{1-r^2}\right)$$ 
with 
$$\nu (s) := e^{-s}\sqrt {1-2e^{-2s}(s^3 - s^4)}.$$
It then suffices to show that $\lim \limits_{s\rightarrow \infty} \frac {d^n}{ds^n} \nu (s) = 0$ for all $n\in \mathbb N$ and $(-1)^n \frac {d^n}{ds^n} \nu (s)$ is non-negative on $(s_n, \infty)$ for some $s_n > 0$ for all $n\in \mathbb N.$ 
This easily follows from the following expression of $\frac {d^n}{ds^n} \nu (s)$,
\begin{displaymath}
\frac {d^n}{ds^n} \nu (s) = e^{-s}\frac {(-1)^n + \sum_{k=1}^n e^{-ks}P_k (s)}{(1-2e^{-2s}(s^3 - s^4))^{n-\frac 12}},
\end{displaymath}
where $P_k$ is a polynomial. This expression is proved by an induction argument. For the general case of $A,B$ and $\alpha$ we get a similar pattern with an exponential term times some \textit{lower order} terms.

This means $\mu_0$ is a radial weight that decays exponentially on the boundary that satisfies all the other conditions in Theorem \ref{DZ}. We conclude that ${\bf{B}}_0$ is unbounded on $L^p(\disk, \mu_0)$ for all $p \in (1, \infty)$, except $p = 2$. Therefore, the Szeg\"o projection operator $\mathbf{S}_{\Omega}$ is also unbounded on $L^p(b\Omega)$ for all $p \in (1, \infty)$, except $p = 2$.
\end{proof}

\section*{Acknowledgment} 

We thank the referee for helpful comments on the expression of the surface area measure and also for the useful editorial remarks on the exposition of the article.

\end{document}